\documentclass[12pt,reqno]{amsart}
\usepackage{fullpage}
\usepackage{times}
\usepackage{amsmath,amssymb,amsthm,url}
\usepackage[utf8]{inputenc}
\usepackage[english]{babel}
\usepackage{comment}
\usepackage{bbm}
\usepackage{enumerate}
\usepackage{bm}
\usepackage{graphicx}
\usepackage{mathrsfs}
\usepackage[colorlinks=true, pdfstartview=FitH, linkcolor=blue, citecolor=blue, urlcolor=blue]{hyperref}

\newtheorem{thm}{Theorem}[section]
\newtheorem{lem}[thm]{Lemma}
\newtheorem{prop}[thm]{Proposition}
\newtheorem{defn}[thm]{Definition}

\newtheorem{rem}[thm]{Remark}

\newtheorem{conj}[thm]{Conjecture}

\newcommand{\C}{\mathbb{C}} 
\newcommand{\N}{\mathbb{N}}
\newcommand{\Z}{\mathbb{Z}}
\newcommand{\F}{\mathbb{F}}

\begin{document}

\title{Maximality of subfields as cliques in Cayley graphs over finite fields}
\author{Chi Hoi Yip}
\address{Department of Mathematics \\ University of British Columbia \\ 1984 Mathematics Road \\ Vancouver  V6T 1Z2 \\ Canada}
\email{kyleyip@math.ubc.ca}

\maketitle

\begin{abstract}
We show the maximality of subfields as cliques in a special family of Cayley graphs defined on the additive group of a finite field. In particular, this confirms a conjecture of Yip on generalized Paley graphs.
\end{abstract}

\section{Introduction}

Throughout the paper, let $p$ be an odd prime and $q$ a power of $p$. Let $\F_q$ be the finite field with $q$ elements, $\F_q^+$ be its additive group, and $\F_q^*=\F_q \setminus \{0\}$ be its multiplicative group. 

In this paper we study maximal cliques in Cayley graphs. We begin by recalling some basic terminologies. Given an abelian group $G$ and a connection set $S \subset G \setminus \{0\}$ with $S=-S$, the {\em Cayley graph} $\operatorname{Cay}(G,S)$ is
the undirected graph whose vertices are elements of $G$, such that two vertices $g$ and $h$ are adjacent if and only if $g-h \in S$. A {\em clique} in a graph $X$ is a subset of vertices in $X$ such that every two distinct vertices in the clique are adjacent. The clique number of $X$, denoted by $\omega(X)$, is the size of a maximum clique in $X$. A {\em maximal clique} is a clique that is not contained in a strictly larger clique. 

Generalized Paley graphs are well-studied Cayley graphs. They were first introduced by Cohen \cite{SC} in 1988, and have been reintroduced by several groups of authors. Let $d>1$ be a positive integer and $q \equiv 1 \pmod {2d}$. The {\em $d$-Paley graph} on $\F_q$, denoted $GP(q,d)$, is the Cayley graph $\operatorname{Cay}(\F_q^+;(\F_q^*)^d)$, where $(\F_q^*)^d$ is the set of $d$-th powers in $\F_q^*$. Note that the condition $q \equiv 1 \pmod {2d}$ avoids degeneracy of the graph; see for example \cite[Section 4]{SC}. Note that Paley graphs are simply $2$-Paley graphs. $3$-Paley graphs and $4$-Paley graphs are also known as \emph{cubic Paley graphs} and \emph{quadruple Paley graphs}. 

It is known that in the Paley graph $GP(q^2,2)$, the subfield $\F_q$ forms a maximal clique for a trivial reason \cite{BDR}: the clique number of $GP(q^2,2)$ is $q$. In general, Broere,  Döman, and Ridley \cite{BDR} observed that in the generalized Paley graph $GP(q^n,d)$, the subfield $\F_q$ forms a clique if $d \mid \frac{q^n-1}{q-1}$. In this case, this observation leads to  $\omega(GP(q^n,d)) \geq q$, which is much better than the generic best-known lower bound $O(\log q)$ on the clique number that holds for all generalized Paley graphs due to Cohen \cite{SC}. In fact, 
Green \cite{Green} showed that the clique number of almost all Cayley graphs defined on a cyclic group $G$ is $O(\log |G|)$, as $|G| \to \infty$. While in the case of generalized Paley graphs, the underlying group may not be cyclic, Green's result still suggests that a clique in a Cayley graph with exceptional size (that is, much larger than $O(\log q)$) tends to have special algebraic structures. 

Determining the clique number of (generalized) Paley graphs is widely open in general \cite{CL}; we refer to \cite[Section 1.3] {YipG} for a survey on recent (minor) improvements on the square root trivial upper bound. Thus, it is interesting if one can show that the above subfield constructions of cliques are not maximal so that the lower bound on the clique number can be further improved. However, the rigid algebraic structure of subfields suggests that it is very unlikely. Indeed, in \cite[Conjecture 1.4]{Y22}, Yip conjectured that such constructions give rise to maximal cliques:

\begin{conj}[{\cite[Conjecture 1.4]{Y22}}]\label{maxconj}
Let $d>1$ be an integer. Let $q \equiv 1 \pmod {2d}$ be a power of a prime $p$, and let $r$ be the largest integer such that $d \mid \frac{q-1}{p^r-1}$. Then the subfield $\F_{p^r}$ forms a maximal clique in $GP(q,d)$.
\end{conj}

In other words, the conjecture states that if $\F_{p^r}$ is the maximum subfield of $\F_q$ that forms a clique in $GP(q,d)$, then in fact it forms a maximal clique. The motivation of Conjecture~\ref{maxconj} is explained in \cite{Y22} in greater details; in particular, Yip \cite[Section 3]{Y22} showed that if $\F_{p^r}$ is not maximal, then there is a clique that is a $2$-dimensional space over $\F_{p^r}$ and consequently the lower bound on the clique number can be improved significantly to $\omega(GP(q,d)) \geq p^{2r}$. This observation, together with known upper bounds on the clique number, allows Yip \cite[Theorem 1.5 and Theorem 1.6]{Y22} to confirm Conjecture~\ref{maxconj} for cubic Paley graphs with cubic order and quadruple Paley graphs with quartic order. However, a similar argument fails to work in general since the best-known upper bound on the clique number is $O(\sqrt{q})$.

In this paper, we use different ideas to resolve Conjecture~\ref{maxconj}. For simplicity, we call a clique $C$ in a Cayley graph $X=\operatorname{Cay}(\F_q^+,S)$ to be a \emph{subfield clique} if $C$ is a subfield of $\F_{q}$, and we say $C$ is a \emph{maximal subfield clique} if $C$ is not contained in a strictly larger subfield clique. Our first main result confirms Conjecture~\ref{maxconj} in a stronger form: a maximal subfield clique in a generalized Paley graph is a maximal clique. 

\begin{thm}\label{main1}
Let $d>1$ be an integer. Let $q$ be a prime power such that $q^n \equiv 1 \pmod {2d}$ and $q>(n-1)^2$. If $\F_q$ is a maximal subfield clique in $GP(q^n,d)$, then $\F_q$ is also a maximal clique. 
\end{thm}

In \cite[Theorem 1.7]{Y22}, Yip described a similar phenomenon in Peisert graphs and conjectured that $\F_q$ forms a maximal clique in a Peisert graph with order $q^4$ provided that $q>3$; this was confirmed by Asgarli and Yip in \cite[Theorem 1.5]{AY}. Moreover, in \cite[Section 5]{AY} of the same paper, they observed that a similar result holds for generalized Peisert graphs under extra assumptions. 

Our second main result improves and extends the results in \cite[Section 5]{AY} substantially. Before stating that, we shall recall the definition of generalized Peisert graphs, first introduced by Mullin~\cite{NM}. This definition is motivated by the similarity between generalized Paley graphs and Peisert graphs (first introduced by Peisert in \cite{WP2} in order to classify self-complementary symmetric graphs).

\begin{defn}[{\cite[Definition 2.11]{AY}}]\rm
Let $d$ be a positive even integer, and $q$ a prime power such that $q \equiv 1 \pmod {2d}$. The {\em $d$-th power Peisert graph of order $q$}, denoted $GP^*(q,d)$, is the Cayley graph $\operatorname{Cay}(\F_q^+, M_{q,d})$, where
$$
M_{q,d}=\bigg\{g^{dk+j}: 0\leq j \leq \frac{d}{2}-1, k \in \Z\bigg\},
$$
and $g$ is a primitive root of $\F_q$.
\end{defn}

While the definition of $GP^*(q,d)$ depends on the choice of the primitive root $g$, it is clear that the isomorphism class of $GP^*(q,d)$ is independent of the choice of $g$. We refer to \cite[Remark 2.12]{AY} for a discussion on the connection between generalized Peisert graphs and generalized Paley graphs. In particular, if $d$ is even, then $GP^*(q,d)$ contains $GP(q,d)$ as a subgraph and thus the structure of maximal cliques in $GP^*(q,d)$ is potentially richer. However, our second main result shows that a maximal subfield clique in $GP^*(q,d)$ is still a maximal clique. 

\begin{thm}\label{main2}
Let $d \geq 4$ be an even integer. Let $q$ be a prime power such that $q^n \equiv 1 \pmod {2d}$ and $q>(n-1)^2d^4/\pi^2(d-1)^2$. If $\F_q$ is a maximal subfield clique in $GP^*(q^n,d)$, then $\F_q$ is also a maximal clique. 
\end{thm}

Note that Theorem~\ref{main2} refines Theorem~\ref{main1} provided that $q$ is sufficiently large. In fact, we will prove a more general (yet technical) statement for any Cayley graph containing a generalized Paley graph as a subgraph in Proposition~\ref{mainprop}. Before proving our main results, we shall introduce some preliminary tools in Section~\ref{prelim}.

\section{Preliminaries}\label{prelim}
A {\em multiplicative character} of $\F_q$ is a group homomorphism from $\F_q^*$ to the multiplicative group of complex numbers with modulus 1. For a multiplicative character $\chi$, its order $d$ is the smallest positive integer such that $\chi^d=\chi_0$, where $\chi_0$ is the trivial multiplicative character of $\F_q$. We refer to \cite[Chapter 5]{LN} for a general discussion on estimates on character sums. The following theorem, due to Katz \cite{K89}, is crucial in our proofs.

\begin{thm}[Katz]\label{Katz}
Let $\theta \in \F_{q^n}$ such that $\F_q(\theta)=\F_{q^n}$. Let $\chi$ be a non-trivial multiplicative character of $\mathbb{F}_{q^{n}}.$ Then
$$
\left|\sum_{a \in \mathbb{F}_{q}} \chi(\theta+a)\right| \leq(n-1) \sqrt{q}.
$$
\end{thm}

The following definition is helpful for our discussions.

\begin{defn}[{\cite[Definition 2.16]{AY}}]\label{defn:epsilon-bounded}\rm
Let $\varepsilon>0$. A set $M \subset \C$ is said to be $\varepsilon$-lower bounded if for every integer $k \in \N$, and for every choice of $x_1,x_2,\ldots, x_k \in M$, we have
\begin{align*}
\bigg|\sum_{j=1}^{k} x_j\bigg| \geq \varepsilon k.
\end{align*}
\end{defn}

Using trigonometric manipulations, it is not difficult to show the following lemma. 

\begin{lem}[{\cite[Lemma 4.5]{AY}}]\label{GP*lem} 
Let $d \geq 4$ be an even integer, and $\omega=\exp(2\pi i/d)$. Then the set $M=\{\omega^j :0 \leq j \leq d/2-1\}$ is $\left(\frac{\pi}{d}-\frac{\pi}{d^2}\right)$-lower bounded.
\end{lem}

\section{Proof of main results}

We will prove a more general statement in the following proposition, and then deduce Theorem~\ref{main1} and Theorem~\ref{main2} as special cases.

\begin{prop}\label{mainprop}
Let $n \geq 2$ be an integer and $\varepsilon>0$ a real number. Let $X=\operatorname{Cay}(\F_{q^n}^+, S)$ be a Cayley graph with $q>(n-1)^2/\epsilon^2$. Assume that there is an integer $d>1$, such that $X$ contains $GP(q^n,d)$ as a subgraph and the set $M=\{\chi(x): x \in S\}$ is $\varepsilon$-lower bounded for some multiplicative character $\chi$ of $\F_{q^n}$ with order $d$. If $\F_q$ is a maximal subfield clique in $X$, then $\F_q$ is also a maximal clique in $X$.
\end{prop}
\begin{proof}
Assume that $\F_q$ is not a maximal clique; then we can find $\theta \in \F_{q^n} \setminus \F_q$ such that $\F_q \cup \{\theta\}$ forms a clique in $X$. Thus, by definition, for any $a \in \F_q$, we have $\theta-a \in S$ and $\chi(\theta-a) \in M$.

Let $\F_{q^m}$ be the smallest extension of $\F_q$ that contains $\theta$; then $\F_{q^m}$ is necessarily a subfield of $\F_{q^n}$ and $m>1$. Let $\chi'$ be the restriction of $\chi$ on the subfield $\F_{q^m}$; then $\chi'$ is a multiplicative character of $\F_{q^m}$. 

Suppose $\chi'$ is the trivial multiplicative character of $\F_{q^m}$, then $\chi(x)=1$ for each $x \in \F_{q^m}^*$. This means that each element in $\F_{q^m}^*$ is a $d$-th power in $\F_{q^n}^*$ and it follows that $\F_{q^m}^* \subset S$ since $X$ contains $GP(q^n,d)$ as a subgraph. In particular, $\F_{q^m}$ is a subfield clique in $X$ that is strictly larger than $\F_q$, violating the assumption. Thus, $\chi'$ is a non-trivial multiplicative character of $\F_{q^m}$. 

Applying Theorem~\ref{Katz} to the character $\chi'$ on the affine line $\theta+\F_q$ and using the definition that $M$ is $\epsilon$-lower bounded, we obtain that
$$
\epsilon q \leq \left|\sum_{a \in \mathbb{F}_{q}} \chi(\theta+a)\right| =\left|\sum_{a \in \mathbb{F}_{q}} \chi'(\theta+a)\right|\leq(m-1) \sqrt{q} \leq (n-1)\sqrt{q}.
$$
Therefore, $q \leq (n-1)^2/\epsilon^2$, contradicting our assumption. This shows that $\F_q$ is a maximal clique in $X$.
\end{proof}

\begin{rem}\rm
From the proof, it is easy to see that the condition ``$\{\chi(x): x \in S\}$ is $\varepsilon$-lower bounded" in the statement of the above proposition can be weakened to $|\sum_{x \in A} \chi(x)| \geq \epsilon q$
for any $A \subset S$ with $|A|=q$. In other words, if there is $S' \subset S$ such that $|S \setminus S'|$ is small and $\{\chi(x): x \in S'\}$ is $\varepsilon$-lower bounded, then we can still conclude that $\F_q$ is maximal clique provided that $q$ is sufficiently large.
\end{rem}

\begin{rem}\rm
A result of a similar flavor has appeared in \cite[Theorem 1.3]{AY} in terms of maximum cliques in the so-called Peisert-type graphs. It generalizes the celebrated Van Lint--MacWilliams' conjecture (equivalently, Erd\H{o}s-{K}o-{R}ado theorem for Paley graphs of square order), first proved by Blokhuis \cite{Blo84}. We refer to \cite[Section 2]{AY} for a historical discussion.
\end{rem}

Finally, we prove Theorem~\ref{main1} and Theorem~\ref{main2}, and discuss the sharpness of the assumption that $q$ is sufficiently large in both theorems.

\begin{proof}[Proof of Theorem~\ref{main1}]
Note that the connection set $S$ of $GP(q^n,d)$ consists of $d$-th powers in $\F_{q^n}^*$. It follows that $M=\{\chi(x): x \in S\}=\{1\}$ is $1$-lower bounded for any multiplicative character $\chi$ of $\F_{q^n}$ with order $d$. Thus, the theorem follows immediately from Proposition~\ref{mainprop}.
\end{proof}

\begin{rem}\rm
We conjecture that the condition $q>(n-1)^2$ in Theorem~\ref{main1} can be dropped, however we do not know how to remove this condition. When $n \leq 5$, we verified that Theorem~\ref{main1} holds for all $q \leq (n-1)^2$ by enumerating all possible generalized Paley graphs via SageMath. We also verified that Theorem~\ref{main1} holds for all $q \leq 17$ when $n=6$. 
\end{rem}

\begin{proof}[Proof of Theorem~\ref{main2}]
Let $g$ be the primitive root of $\F_{q^n}$ that defines the graph $GP^*(q^n,d)$. Let $\chi$ be a multiplicative character in $\F_q$ such that $\chi(g)=\omega$, where $\omega=\exp(2\pi i/d)$; then $\chi$ has order $d$. As discussed before, $GP^*(q^n,d)$ contains $GP(q^n,d)$ as a subgraph.  Let $M=\{\chi(x): x \in S\}$, where $S=\{g^{j+kd}: 0 \leq j \leq d/2-1, k \in \Z\}$ is the connection set of $GP^*(q^n,d)$. It follows from Lemma~\ref{GP*lem} that $M=\{\omega^j :0 \leq j \leq d/2-1\}$ is $\left(\frac{\pi}{d}-\frac{\pi}{d^2}\right)$-lower bounded. Therefore, by Proposition~\ref{mainprop}, $\F_q$ is a maximal clique provided that
$$
q>\frac{(n-1)^2}{\left(\frac{\pi}{d}-\frac{\pi}{d^2}\right)^2}=\frac{(n-1)^2d^4}{\pi^2(d-1)^2}.
$$
\end{proof}
\begin{rem}\rm
We believe that the condition $q>(n-1)^2d^4/\pi^2(d-1)^2$ in Theorem~\ref{main2} is not optimal. However, we do need to assume $q$ is sufficiently large for Theorem~\ref{main2} to hold. There are plenty of counterexamples when $q$ is small compared to $n$ and $d$. For example, when $q=3,n=4$, and $d=4$, the subfield $\F_3$ is a maximal subfield clique in $GP^*(81,4)$, and yet there is a maximal clique with size $9$ containing $\F_3$. Similarly, when $q=5,n=6$, and $d=62$, the subfield $\F_5$ is a maximal subfield clique in $GP^*(15625,62)$, and yet there is a maximal clique with size $25$ containing $\F_5$.
\end{rem}

\section*{Acknowledgement}
The author thanks Shamil Asgarli for many helpful discussions. The research of the author is supported by a doctoral fellowship from the University of British Columbia.

\end{document}